\documentclass[a4paper,10pt]{article}

\usepackage[utf8]{inputenc}   
\usepackage[T1]{fontenc} 
\usepackage{amsmath,amsthm}
\usepackage[english]{babel}
\usepackage[dvips]{graphicx}
\usepackage{amsfonts,amssymb}
\usepackage{geometry}
\usepackage{hyperref}
\usepackage{stmaryrd}
\usepackage{fancyhdr}
\usepackage{url}

\newtheorem*{conj*}{Conjecture}
\newtheorem*{claim}{Claim}
\newtheorem{prop}{Proposition}[section]
\newtheorem{LM}{Lemma}[section]
\newtheorem*{st}{}
\newtheorem{thm}{Theorem}[section]

\newtheorem{cor}{Corollary}[section]

\newtheoremstyle{pourlesremarques}{\topsep}{\topsep}{\normalfont}{}{\bfseries}{.}{ }{}
\theoremstyle{pourlesremarques}
\newtheorem{rem}{Remark}[section]
\newtheoremstyle{pourlesexemples}{\topsep}{\topsep}{\normalfont}{}{\bfseries}{.}{ }{}
\theoremstyle{pourlesexemples}

\renewcommand{\o}{\mathfrak{O}_F}
\newcommand{\p}{\mathfrak{P}_F}
\newcommand{\w}{\varpi_F}
\renewcommand{\d}{\delta}

\title {\textbf{Derivatives and asymptotics of Whittaker functions}}
\author{Nadir MATRINGE\footnote{Nadir Matringe, University of East Anglia, School of Mathematics, Norwich, UK, NR4 7TJ. Email: n.matringe@uea.ac.uk}}

\begin{document}
\maketitle

\begin{abstract}
Let $F$ be a $p$-adic field, and $G_n$ one of the groups $GL(n,F)$, $GSO(2n-1,F)$, $GSp(2n,F)$, or $GSO(2(n-1),F)$. Using the mirabolic subgroup or 
analogues of it, and related ``derivative'' functors, we give an asymptotic expansion of functions in the Whittaker model of generic representations 
of $G_n$, with respect to a minimal set of characters of subgroups of the maximal torus. Denoting by $Z_n$ the center of $G_n$, and by $N_n$
 the unipotent radical of its standard Borel subgroup, we characterize generic representations occurring in $L^2(Z_nN_n\backslash G_n)$ in terms of these characters.\\ 
This is related to a conjecture of Lapid and Mao for general split groups, asserting that the generic representations occurring in $L^2(Z_nN_n\backslash G_n)$ are 
the generic discrete series; we prove it for the group $G_n$.
 \end{abstract}

\section*{Introduction}

Let $G_n$ be the points of one of the groups $GL(n)$, $GSO(2n-1)$, $GSp(2n)$, or $GSO(2(n-1))$ over a $p$-adic field $K$. The main result (Theorem \ref{restorus}) 
of this work 
describes the asymptotic behaviour of the restriction of Whittaker functions to the standard maximal torus, in terms of a 
family of characters which is minimal in some sense. From results of \cite{LM}, this restriction can be described for split reductive 
groups in terms of cuspidal exponents.\\ 
Here, after having defined analogues of the mirabolic subgroup for the groups $G_n$, and the corresponding derivative functors, following \cite{CP} 
(where the case of completely reducible derivatives is treated for $GL(n)$), we choose to describe the restriction of Whittaker functions to the torus
 in terms of central exponents of the derivatives.\\ 
This description, inspired by \cite{B}, is better
adapted to understanding when the Whittaker model of a unitary generic representation is a subspace of 
$L^2(Z_n N_n\backslash G_n)$ (these representations are conjectured to be generic discrete series by Lapid and Mao).\\ 
In the first section, we review the groups in question and define their mirabolic subgroups. We also give a decomposition of the unipotent
 radical of the 
standard Borel subgroup, and a description of how nondegenerate characters of this radical behave with respect to this decomposition.\\
In Section 2, we give properties of the derivative functors, and use them to prove our asymptotic expansion of Whittaker functions, which is 
Theorem \ref{restorus}.\\
In Section 3, we characterize generic representations with
Whittaker model included in $L^2(Z_n N_n\backslash G_n)$ in terms of central exponents of the derivatives, in Corollary \ref{l2g}. We then 
prove in Theorem \ref{l2'}
the conjecture 3.5 of \cite{LM}.\\
After the writing of this paper, Patrick Delorme informed us that he obtained 
the proof of the conjecture of Lapid and Mao in the general case (Theorem 8 of \cite{D2}).

\section{Mirabolic subgroup and nondegenerate characters}

Let $F$ be nonarchimedean local field, we denote by $\o$ its ring of integer, and by 
$\p=\w\o$ the maximal ideal of this ring, where $\w$ is a uniformiser of $F$.\\ 
We give a list of groups, and describe some of their properties which will be used in the sequel:\\

\begin{itemize}
\item \textbf{Case A:}\\
The group $G_0$ is trivial, and for $n\geq1$, the group $G_n$ is $GL(n,F)$. \\

We consider the maximal torus of $G_{n}$ consisting of diagonal matrices, it is isomorphic to $(F^*)^{n}$.\\
For $n\geq2$, the simple roots of this group can be chosen to be the characters $$\alpha_{i}(diag(x_1,\dots,x_n))= x_i x_{i+1}^{-1}$$ 
for $i$ between $1$ and $n-1$.\\
The root subgroup $U_{\alpha_i}$ is given by matrices of the form $I_n + xE_{i,i+1}$ for $x$ in $F$.\\

Standard Levi subgroups of $G_n$
 are given by matrices of the form 
$diag(a_1,\dots,a_r)$ 
where $a_i$ belongs to $GL(n_i,F)$, with $n_1+\dots + n_r=n$. We denote the preceding group by $M_{(n_1,\dots,n_r)}$, and the corresponding standard 
parabolic subgroup is denoted by
 $P_{(n_1,\dots,n_r)}$, with unipotent radical $U_{(n_1,\dots,n_r)}$.\\
 
For $n\geq2$ we denote by $U_{n}$ the group $U_{(n-1,1)}$, of matrices of the form $\begin{bmatrix} 
                                                                                         I_{n-1}    & V\\
                                                                                                    & 1 \end{bmatrix}$.

It is isomorphic to $F^{n-1}$.\\

For $n>k\geq 1$, the group $G_{k}$ embeds naturally in $G_{n}$, and is given by matrices of the form 
$diag(g,I_{n-k})$; we denote by $Z_k$ its center.\\
We denote by $P_{n}$ the mirabolic subgroup $G_{n-1}\ltimes U_{n}$.

\item \textbf{Case B:}\\
The group $G_0$ is trivial.\\
For $n\geq2$, the group $G_n=GO(2n-1,F)$ is the group of matrices $g$ in $GL(2n-1,F)$ such that $^t\!g J g$ belongs to $F^*J$. 
We call the multiplier of an element $g$ in $G_n$ the scalar $\mu(g)$ such that $^t\!g J g$ is equal to 
$\mu(g)J$ (one checks that for this group, the multiplier actually belongs to $(F^*)^2$), where 
$J$ is the antidiagonal matrix of $GL(2n-1,F)$ with ones on the second diagonal. It is the direct product of 
$SO(2n-1,F)$ with $F^*$, more precisely, of $SO(2n-1,F)$ and the group $I(F^*)$ of matrices $I(t)=tI_{2n-1}$ for $t$ in $F^*$.\\

The maximal torus of $G_{n}$ is equal to the product of the torus of matrices of the form $diag(x_{n-1}^{-1},\dots,x_1^{-1},1,x_1,\dots,x_{n-1})$
 with $I(F^*)$ and is isomorphic to $(F^*)^{n}$.\\
For $n\geq3$, the simple roots of this group can be chosen to be the characters 
$$\alpha_{i+1}(diag(tx_{n-1}^{-1},\dots,tx_1^{-1},t,tx_1,\dots,tx_{n-1}))=x_i x_{i+1}^{-1}$$ 
for $i$ between $1$ and $n-2$, and 
$\alpha_1(diag(tx_{n-1}^{-1},\dots,tx_1^{-1},t,tx_1,\dots,tx_{n-1}))= x_1^{-1}$.\\
The root subgroups $U_{\alpha_i}$ are given by matrices of the form $$diag(u,\dots,1,u,1,\dots,1,u^{-1},1,\dots,1)$$
 for matrices $u$ in the unipotent radical
 of the Borel of $GL(2,F)$.\\

For $n > k \geq 1$, the group $G_{k}$ embeds naturally in $G_n$, and is given by matrices of the form 
$diag(\mu(g)I_{k},g,I_{n-k})$, where $\mu(g)$ 
is the multiplier of the element $g$ of $G_k$.\\
 Its center $Z_k$ is given by matrices of the form 
$z_k(t)=diag( t^2 I_{n-k},tI_{2k-1},I_{n-k})$ for $t$ in $K^*$. \\

For $n\geq1$ , the standard Levi subgroups of $G_n$ are given by matrices of the form 
$$diag(\mu(g)^\tau\!a_{r-1}^{-1},\dots,\mu(g)^\tau\!a_1^{-1},g,a_1,\dots,a_r),$$
where $a_i$ belongs to $GL(n_i,F)$, $^\tau\!a$ is the transpose of $a$ with respect to the second diagonal, $g$ belongs to $G_m$,
 with $2m-1+2n_1+\dots + 2n_r=2n-1$. We denote the preceding group by $M_{(m;n_1,\dots,n_r)}$, and the corresponding standard parabolic subgroup consisting 
 of block upper triangular matrices is denoted by
 $P_{(m;n_1,\dots,n_r)}$, with unipotent radical $U_{(m;n_1,\dots,n_r)}$.\\

For $n\geq2$ we denote by $U_{n}$ the group $U_{(n-1;1)}$, of matrices of the form $\begin{bmatrix} 1 & -^\tau\!V  &  -^\tau\!V V/2 \\
                                                                                       &   I_{2n-3}    & V\\
                                                                                       &              & 1 \end{bmatrix}$.\\

It is isomorphic to $F^{2n-3}$.\\

For $n\geq2$, we denote by $P_n$ the ``mirabolic'' subgroup $G_{n-1}\ltimes U_{n}$.

\item \textbf{Case C:}\\
The group $G_0$ is trivial, the group $G_1$ is $F^*$.\\
For $n\geq2$, the group $G_n=GSp(2(n-1),F)$, where $GSp(2(n-1),F)$ is the group of matrices $g$ in $GL(2(n-1),F)$ such that $^t\!g J g$ belongs to $F^*J$, where 
$J=\begin{bmatrix} 0 & W \\ -W & 0 \end{bmatrix}$ and $W$ is the antidiagonal matrix of $GL(n-1,F)$ with ones on the second diagonal. It is the semi-direct product of 
$Sp(2(n-1),F)$ with $F^*$, more precisely, of $Sp(2n,F)$ and the group $I(F^*)$ of matrices $I(t)=diag(tI_{n-1},I_{n-1})$ for $t$ in $F^*$.\\

The maximal torus of $G_{n}$ is equal to the product of the torus of matrices of the form $diag(x_{n-1}^{-1},\dots,x_1^{-1},x_1,\dots,x_{n-1})$
 with $I(F^*)$ and is isomorphic to $(F^*)^{n}$.\\
The simple roots of this group are the characters $\alpha_{i+1}(diag(tx_{n-1}^{-1},\dots,tx_1^{-1},x_1,\dots,x_{n-1}))=x_i x_{i+1}^{-1}$ for $i$ between $1$ and $n-2$, and 
$\alpha_1(diag(tx_{n-1}^{-1},\dots,tx_1^{-1}, x_1,\dots, x_{n-1}))= tx_1^{-2} $.\\
For $i$ less than $n$, the root subgroup $U_{\alpha_i}$ is given by matrices of the form $$diag(1,\dots,1,u,1,\dots,1,u^{-1},1,\dots,1),$$
 for matrices $u$ in the unipotent radical
 of the Borel of $GL(2,F)$, whereas $U_{\alpha_n}$ is given by matrices $diag(1,\dots,1,u,1,\dots,1)$, for matrices $u$ in the unipotent radical
 of the Borel of $GL(2,F)$.\\ 

For $n > k \geq 2$, the group $G_{k}$ embeds naturally in $G_n$, and is given by matrices of the form 
$diag(\mu(g)I_{n-k},g,I_{n-k})$, where $\mu(g)$ 
is the multiplier of the element $g$ of $G_k$. Its center $Z_k$ is given by matrices of the form 
$z_k(t)=diag( t^2 I_{n-k},tI_{2(k-1)},I_{n-k})$ for $t$ in $F^*$. The group $G_1$, which is equal to its center $Z_1$, embeds as $I(F^*)$.\\

For $n\geq1$, the standard Levi subgroups of $G_n$ are:
\begin{itemize}
\item either matrices $$diag(\mu(g)^\tau\!a_{r-1}^{-1},\dots,\mu(g)^\tau\!a_1^{-1},g,a_1,\dots,a_r),$$

where $a_i$ belongs to $GL(n_i,F)$, $g$ belongs to $G_m$ with $m\geq2$,
 with $2(m-1)+2n_1+\dots + 2n_r=2(n-1)$. We denote the preceding group by $M_{(m;n_1,\dots,n_r)}$, and the corresponding standard parabolic subgroup
  consisting of block upper triangular matrices is denoted by
 $P_{(m;n_1,\dots,n_r)}$, with unipotent radical $U_{(m;n_1,\dots,n_r)}$.\\

\item or the matrices $$z_1.diag(a_{r-1}^{-1},\dots,a_1^{-1},a_1,\dots,a_r),$$ with $a_i$ in $GL(n_i,F)$, $2n_1+\dots + 2n_r=2(n-1)$, and $z_1$ in $G_1$.
\end{itemize}

 
For $n\geq 2$, we denote by $U_n$ the group $U_{(n-1;1)}$, of matrices of the form $\begin{bmatrix} 1 & -^\tau\!V_2  & ^\tau\!V_1  & x \\
                                                                                       &   I_{n-2}    &             & V_1\\
                                                                                       &              &   I_{n-2}   & V_2\\ 
                                                                                       &              &             & 1 \end{bmatrix}$.\\

For $n\geq3$, it is an extension of $F^{2(n-2)}$ by $F$, which is the Heisenberg group corresponding to the alternating bilinear form on $F^{2(n-2)}$,
 given by $(W_1,W_2)\times (V_1,V_2)\mapsto  -^\tau\!W_2 .V_1 +  ^\tau\!W_1. V_2$.\\
 It is a two steps nilpotent subgroup, with center equal to its derived subgroup,
 given by matrices with $V_1=V_2=0$, and the maximal abelian quotient $U_{n}^{ab}$ of $U_{n}$ is $F^{2(n-2)}$.\\
The group $U_{2}$ is the unipotent radical of the standard Borel of $G_2=GSp(2,F)=GL(2,F)$, and is isomorphic to $(F,+)$.\\

For $n\geq 2$, we denote by $P_{n}$ the ``mirabolic'' subgroup $G_{n-1}\ltimes U_{n}$.

\item \textbf{Case D:}\\
We denote by $G_0$ the trivial group. The group $G_1$ is $F^*$.\\

For $n\geq 2$, the group $G_{n}=GSO(2(n-1),F)$ is the group of matrices $g$ in $GL(2n,F)$ satisfying that $^t\!g J g$ belongs to $F^*J$, where 
$J$ is the antidiagonal matrix of $GL(2n,F)$ with ones an the second diagonal. It is the semi-direct product of 
$SO(2(n-1),F)$ with the group $I(F^*)$, of the matrices $I(t)=diag(tI_n,I_n)$ for $t$ in $F^*$.\\

The maximal torus of $G_{n}$ is equal to the product of the torus of matrices of the form $diag(x_{n-1}^{-1},\dots,x_1^{-1},x_1,\dots,x_{n-1})$
 with $I(F^*)$ and is isomorphic to $(F^*)^{n}$.\\
If $n$ is $2$, then $G_2$ is the diagonal torus of $GL(2,F)$.\\
For $n\geq 3$, the simple roots of this group can be chosen to be the characters $$\alpha_{i+1}(diag(tx_{n-1}^{-1},\dots,tx_1^{-1},x_1,\dots,x_{n-1}))=x_i x_{i+1}^{-1}$$ 
for $i$ between $1$ and $n-2$, and 
$\alpha_1(diag(tx_{n}^{-1},\dots,tx_1^{-1},x_1,\dots,x_{n}))= tx_1^{-1}x_2^{-1} $.\\
For $i\geq1$, the root subgroup $U_{\alpha_{i+1}}$ is given by matrices of the form $$diag(1,\dots,1,u,1,\dots,1,u^{-1},1,\dots,1)$$
 for matrices $u$ in the unipotent radical
 of the Borel of $GL(2,F)$, whereas $U_{\alpha_1}$ is given by matrices $diag(1,\dots,1,u,1,\dots,1)$ for matrices $u$ of the form 
$\begin{bmatrix} 1 &  & y  & \\  & 1 &   & -y \\  &  & 1 &  \\  &  &  & 1 \end{bmatrix}$, with $y$ in $F$.\\ 

For $n > k \geq 2$, the group $G_{k}$ embeds naturally in $G_n$, and is given by matrices of the form 
$diag(\mu(g)I_{n-k},g,I_{n-k})$, where $\mu(g)$ 
is the multiplier of the element $g$ of $G_k$. For $k\geq3$, its center $Z_k$ is given by matrices of the form 
$z_k(t)=diag( t^2 I_{n-k},tI_{2(k-1)},I_{n-k})$ for $t$ in $F^*$.\\
\textbf{We denote by $\mathbf{Z_2}$ the subgroup of the torus $\mathbf{G_2}$, given by matrices of the form 
$\mathbf{z_2(t)=diag( t I_{n-2},1,t,I_{n-2})}$ for $\mathbf{t}$ in $\mathbf{F^*}$}.\\
The group $G_1$ which is equal to its center $Z_1$, embeds as $I(F^*)$.\\

The standard Levi subgroups of $G_{n}$ are the following:\\
\begin{itemize}
 \item The groups given by matrices of the form 
$$diag(\mu(g)^\tau\!a_{r}^{-1},\dots,\mu(g)^\tau\!a_1^{-1},g,a_1,\dots,a_r),$$
where $a_i$ belongs to $GL(n_i,F)$, $g$ belongs to $G_m$ with $m\geq 3$,
 with $2(m-1)+2n_1+\dots + 2n_r=2(n-1)$. We denote the preceding group by $M_{(m;n_1,\dots,n_r)}$, and the corresponding standard parabolic subgroup is denoted by
 $P_{(m;n_1,\dots,n_r)}$, with unipotent radical $U_{(m;n_1,\dots,n_r)}$.
 
\item The groups given by matrices of the form 
$$g_2.diag(^\tau\! a_{r}^{-1},\dots,^\tau\!a_1^{-1},1,1,a_1,\dots,a_r),$$ with $a_i$ in $GL(n_i,F)$, $2n_1+\dots + 2n_r=2(n-2)$, and $g_2$ in $G_2$.

\item The groups given by matrices of the form 
$$z_1.diag(^\tau\!a_{r}^{-1},\dots,^\tau\!a_1^{-1},a_1,\dots,a_r),$$ with $a_i$ in $GL(n_i,F)$, $2n_1+\dots + 2n_r=2(n-1)$, and $z_1$ in $G_1$.

 \item The groups given by matrices of the form $$z_2.diag(^\tau\!a_{r}^{-1},\dots,^\tau\!a_1^{-1},g,a_1,\dots,a_r),$$ where $a_i$ belongs to $GL(n_i,F)$,
and $g$ in $GL(2(m-1),F)$ is of the form $$\begin{bmatrix} A &  &V &  \\  & t'& &L' \\ L &  &t &  \\  & V'& &A'  \end{bmatrix},$$ 
with $$\begin{bmatrix} A & V \\ L & t  \end{bmatrix}\in GL(m-1,F),$$ $A$ a $(m-2\times m-2)$-matrix,
$$\begin{bmatrix} ^\tau\!A' & ^\tau\!L' \\^\tau\! V' & t'  \end{bmatrix}=\begin{bmatrix} A & V \\ L & t  \end{bmatrix}^{-1},$$
 with $2m+2n_1+\dots + 2n_r=2n$ and $z_2$ in $Z_2$.
\end{itemize}

For $n\geq 3$ we denote by $U_{n}$ the subgroup $U_{(n-1;1)}$ of $G_{n}$, of matrices of the form $$\begin{bmatrix} 1 & -^\tau\!V  &  -^\tau\!V V/2 \\
                                                                                       &   I_{2(n-2)}    & V\\
                                                                                       &              & 1 \end{bmatrix}.$$

It is isomorphic to $F^{2n-2}$.\\

For $n\geq 3$ denote by $P_{n}$ the ``mirabolic'' subgroup $G_{n-1}\ltimes U_{n}$.\\ 
We denote by $U_2$ the group $U_{\alpha_1}$, and by $P_2$ the group $G_1\ltimes U_2$.

\end{itemize}

\begin{LM}\label{torusdec}
We denote by $Z_i$ the center of $G_i$, except in case $D$ when $n=2$, where we denote by $Z_2$ the subgroup $diag( 1, t)$ with $t$ in $F^*$ of $G_2$.
 In all cases, one checks that the maximal torus $A_n$ of $G_n$ is the direct product $Z_1.Z_2\dots Z_{n-1}.Z_n$, and each $Z_i$ is isomorphic to $F^*$.
 Moreover, the i-th root has the property that $\alpha_i(z_1\dots z_n)=z_i$, in other words these coordinates parametrize the torus 
$A_n$ such that simple roots become canonical projections.
\end{LM}

The unipotent radical $N_{n+1}$ of the standard Borel subgroup of $G_{n+1}$ is equal to $U_2\dots U_{n+1}$. Let $\theta$ be a nondegenerate character of $N_{n+1}$ 
(i.e. that restricts non trivially to any of the simple root subgroups). We denote by $\theta_{i+1}$ the character $\theta_{|U_{i+1}}$, except in the case $D$,
for $i=2$. In this case $U_3=U_{\alpha_1}\times U_{\alpha_2}$, and we denote by $\theta_3$ the character $\theta_3(u_{\alpha_1}u_{\alpha_2})=\theta(u_{\alpha_2})$.\\
 Because $\theta$ is trivial 
on $U_{der}$, and according to the description of $U_{der}$ in Theorem 4.1 of \cite{BH}, the character $\theta_{i+1}$ must be trivial on every root subgroup $U_{\alpha}$ 
contained in $U_{i+1}$ such that $\alpha$ is not simple, moreover for case D, $n=2$, the character $\theta_3$ is trivial on $U_{\alpha_1}$.\\
Conversely, if a non trivial character $\theta_{i+1}$ of $U_{i+1}$ is trivial on every $U_{\alpha}\subset U_{i+1}$ which is not simple, and if, in case D, $n=2$, 
we impose in addition that $\theta_3$ is trivial on $U_{\alpha_1}$,
then one checks that the normalizer of $\theta_{i+1}$ in 
the mirabolic subgroup $P_{i+1}$ is $P_{i }U_{i+1}$. 
As the group $U_2\dots U_{i}$ is a subgroup of $P_{i}$, a family of non trivial characters 
$\theta_{i+1}$ of $U_{i+1}$, trivial on every $U_{\alpha}\subset U_{i+1}$ except $U_{\alpha_i}$, defines a nondegenerate character of $N_{n+1}= U_2\dots U_{n+1}$ by 
$\theta(u_2\dots u_{n+1})=\prod_{i=1}^n \theta_{i+1}(u_{i+1})$.\\

Now we fix such a nondegenerate character $\theta$, and write $\theta^{k}$ for the character $\theta_2\dots\theta_{k}$ of $N_{k}$.\\

\section{Derivatives and Whittaker functions}\label{der}

If $G$ is an $l$-group, we denote by $Alg(G)$ the category of smooth complex $G$-modules. If $(\pi,V)$ belongs to $Alg(G)$, $H$ is a closed subgroup of $G$,
 and $\chi$ is a character of $H$, we denote by $V(H,\chi)$ the subspace of $V$ generated by vectors of the form $\pi(h)v-\chi(h)v$ for $h$ in $H$ and $v$ in $V$. 
This space is actually stable under the action of the subgroup $N_G(\chi)$ of the normalizer $N_G(H)$ of $H$ in $G$, which fixes $\chi$.\\
We denote by $\delta_H$ the positive character of $N_G(H)$ such that if $\mu$ is a right Haar measure on $G$, and $\lambda$ is the left translation 
of smooth functions with compact support on $G$, then $\mu \circ \lambda(n^{-1})= \delta_H(n)\mu $ for $n$ in $N$.\\ 
 This gives the spaces $V(H,\chi)$ and $V_{H,\chi}=V/V(H,\chi)$ (that we simply denote by $V_H$ when 
$\chi$ is trivial) a structure of smooth $N_G(\chi)$-modules.\\  
 
Notations being as in the first section, and for $k$ be an integer between $2$ and $n$ we define the following functors:\\

\begin{itemize}

\item First we recall the definition of the Jacquet functors:\\
Let $P$ be a parabolic subgroup of $G_n$, with Levi subgroup $M$, and unipotent radical $U$.\\ 
We denote by $J_P$ the functor from 
$Alg(G_n)$ to $Alg(M)$ such that, if $(\pi,V)$ is a smooth $G_n$-module, 
we have $J_P(V) =V_{U}$, and $M$ acts on $J_P(V)$ by
 $J_P \pi (m)(v+V(U,1))= \delta_{U} (m)^{-1/2}\pi (m)v+V(U,1)$.\\

\item With the same notations, we denote by $i_P^G$ the functor from $Alg(M)$ to $Alg(G_n)$ such that, 
if $\rho$ is a smooth $M$-module, and $\bar{\rho}$ is the corresponding $P$-module obtained by inflation of $\rho$ to $P$, then $i_P^G(\rho)$ is the 
$G_n$-module $ind_P^{G_n}(\bar{\rho})$ where $ind$ is the usual compact induction.

 \item The functor $\Phi_{\theta_k}^{-}$ (denoted $r_{U_k,\theta_k}$ in section 1 of \cite{BZ2}) from $Alg(P_k)$ to $Alg(P_{k-1})$ such that, if $(\pi,V)$ is a smooth $P_k$-module, 
$\Phi_{\theta_k}^{-} V =V_{U_k,\theta_k}$, and $P_{k-1}$ acts on $\Phi_{\theta_k}^{-}(V)$ by
 $\Phi_{\theta_k}^{-} \pi (p)(v+V(U_k,\theta_k)= \delta_{U_k} (p)^{-1/2}\pi (p)(v+V(U_k,\theta_k)$.

\item The functor $\Phi_{\theta_k}^{+}$ (denoted $i_{U_k,\theta_k}$ in section 1 of \cite{BZ2}) from $Alg(P_{k-1})$ to $Alg(P_{k})$ such that, for $\pi$ in $Alg(P_{k-1})$, one has
$\Phi_{\theta_k}^{+} \pi = ind_{P_{k-1}U_k}^{P_k}(\delta_{U_k}^{1/2}\pi \otimes \theta_k)$, where $ind$ is the usual compact induction.

\item The functor $\hat{\Phi}_{\theta_k}^{+}$ ($I_{U_k,\theta_k}$ in section 1 of \cite{BZ2}) from $Alg(P_{k-1})$ to $Alg(P_{k})$ such that, for $\pi$ in $Alg(P_{k-1})$, one has
$\Phi_{\theta_k}^{+} \pi = Ind_{P_{k-1}U_k}^{P_k}(\delta_{U_k}^{1/2}\pi \otimes \theta_k)$, where $Ind$ is the usual induction.

\item The functor $\Psi^{-}$ is the Jacquet functor $J_{U_k}$, (denoted $r_{U_k,1}$ in section 1 of \cite{BZ2}) from $Alg(P_k)$ to $Alg(G_{k-1})$, such that if $(\pi,V)$ is a smooth $P_k$-module, 
$\Psi^{-} V =V_{U_k,1}$, and $G_{k-1}$ acts on $\Psi^{-}(V)$ by
 $\Psi^{-} \pi (g)(v)+V(U_k,\theta_k)= \delta_{U_k} (g)^{-1/2}\pi (p)(v+V(U_k,1))$.

\item The functor $\Psi^{+}$ (denoted $i_{U_k,1}$ in section 1 of \cite{BZ2}) from $Alg(G_{k-1})$ to $Alg(P_{k})$, such that for $\pi$ in $Alg(G_{k-1})$, one has
$\Psi^{+} \pi = ind_{G_{k-1}U_k}^{P_k}(\delta_{U_k}^{1/2}\pi \otimes 1)=\delta_{U_k}^{1/2}\pi \otimes 1 $.

\end{itemize}

As we already fixed the character $\theta$ of $N_n$, we will most of the time forget the dependence in $\theta_k$ of 
$\Phi_{\theta_k}^{-}$ and $\Phi_{\theta_k}^{+}$, and we will write these functors $\Phi^{-}$ and $\Phi^{+}$. 
These functors have the following properties which follow (except for c) an d) which are trivial) from Proposition 1.9 of \cite{BZ2}:

\begin{prop}
a) The functors $\Phi^{-}$, $\Phi^{+}$, $\Psi^{-}$, and $\Psi^{+}$ are exact.\\
b) $\Psi^{-}$ is left adjoint to $\Psi^{+}$.\\
b') $\Phi^{-}$ is left adjoint to $\hat{\Phi}^{+}$.\\
c) $\Phi^{-}\Psi^{+}=0$\\
d) $\Psi^{-}\Psi^{+}=Id$.
\end{prop}

Now we want to know how these functors restrict to smooth $P_k$-modules which are submodules of the space $C^{\infty}(N_k\backslash P_k,\theta^k)=Ind_{N_k}^{P_k}(\theta^k)$ 
of functions on $P_k$,
 fixed by some open subgroup of $P_k$ under right translation, and which transform by $\theta^k$ under left translation by elements of $N_k$.\\
The next proposition shows the stability of this type of modules under $\Phi^{-}$ and $\Phi^{+}$.

\begin{prop}\label{stab}
For any submodule $\tau$ of $C^{\infty}(N_k\backslash P_k,\theta^k)$, the $P_{k-1}$-module $\Phi^{-}\tau$ is a submodule of 
$C^{\infty}(N_{k-1}\backslash P_{k-1},\theta^{k-1})$,
 with model given by restriction of functions $\delta_{U_k}^{-1/2}W$ in $\tau$ to $P_{k-1}$, and such that we have 
$\Psi^{-}\tau(p)W= \rho (p)W$ for $p$ in $P_{k-1}$, where $\rho$ is the action by right translation.\\ 
Conversely, the $P_{k+1}$-module $\Phi^{+}\tau$ can be identified with a submodule of $C^{\infty}(N_{k+1}\backslash P_{k+1},\theta^{k+1})$, with the natural action
 of $P_{k+1}$ by right translation.
\end{prop}
\begin{proof}
The first property will hold if we show that $C^{\infty}(N_k\backslash P_k,\theta^k)(U_k,\theta_k)$ is  the kernel of the restriction map to
 $C^{\infty}(N_{k-1}\backslash P_{k-1},\theta^{k-1})$, this is a straightforward adaptation of the proof of Proposition 2.1 of 
\cite{CP}.\\
The second property is a consequence of the following equalities and inclusions:\\
$$\begin{array}{cc}
\Phi^{+}(C^{\infty}(N_k\backslash P_k,\theta^k))  &= ind_{P_{k}U_{k+1}}^{P_{k+1}}(\delta_{U_{k+1}}^{1/2}.Ind_{N_k}^{P_k}(\theta^k) \otimes \theta_{k+1})\\
   & \subset Ind_{P_{k}U_{k+1}}^{P_{k+1}}(\delta_{U_{k+1}}^{1/2}.Ind_{N_k}^{P_k}(\theta^k) \otimes \theta_{k+1}) \end{array}.$$

Then $$ \delta_{U_{k+1}}^{1/2}.Ind_{N_k}^{P_k}(\theta^k)\simeq Ind_{N_k}^{P_k}(\theta^k) $$ because the character $\delta_{U_{k+1}}^{1/2}$ of 
$P_{k}$ is trivial on $N_k$.

Finally $$Ind_{P_{k}U_{k+1}}^{P_{k+1}}(Ind_{N_k}^{P_k}(\theta^k)\otimes \theta_{k+1})\simeq Ind_{N_{k+1}}^{P_{k+1}}(\theta^{k+1})$$  

 \end{proof}

More can be said about smooth $P_k$-submodules of the space $C^{\infty}(N_k\backslash P_k,\theta^k)=Ind_{N_k}^{P_k}(\theta^k)$.
If $\tau$ is a $P_k$-submodule of $C^{\infty}(N_k\backslash P_k,\theta^k)$,
 then the derived subgroup of $U_k$ (which is trivial except in case D) acts trivially.\\ 
To see this, take $W$ in $C^{\infty}(N_k\backslash P_k,\theta^k)$, we claim that if $u$ belongs to the derived subgroup $U_k^{der}$ of $U_k$, then $\tau(u)W$ and $W$ are 
equal. So let $p$ belong
 to $P_{k}$; one has $\tau(u)W(p)=W(pu)= W(pup^{-1}p)=\theta^k(pup^{-1})W(p)$. But $P_{k}$ normalizes $U_k$ (so $\theta^k(pup^{-1})=\theta_k(pup^{-1})$), so that 
it normalizes its derived subgroup as well; as $\theta_k$ is trivial on this subgroup, this proves our claim.\\

For such modules $P_k$-modules, there is a nice interpretation of $V(U_k,1)$ in terms of the analytic behaviour of Whittaker functions.
 First, we make the following observation.\\

\begin{rem}\label{vanishlarge}
For $k\geq 3$, as a consequence of the Iwasawa decomposition, any element $g$ of $G_{k-1}$ can be written in the form $pzc$ with $p$ in $P_{k-1}$, 
$z$ in $Z_{k-1}$, 
and $k$ in $K=G_{k-1}(\o)$, and the absolute value of $z$ depends only on $g$, so we can write it $|z(g)|_F$. \\
 If a function $W$ is in the space of $C^{\infty}(N_k\backslash P_k,\theta^k)$, then for $g$ in $G_{k-1}$, we show that 
 $W(g)$ vanishes whenever $|z(g)|_F$ is large enough.\\ 
Indeed if we take the ``natural'' group isomorphism $u$ from $(F^m,+)$ to $U_k^{ab}$, for some positive integer $m$, and recalling that it is in 
fact $U_k^{ab}$ that acts on $V$, 
then $u(x)$ will fix $W$ for $x$ near zero in $F^m$.\\
 But then, for $g$ in $G_{k-1}$ of the form $pzk$, one has
 $W(g)=W(gu(x))=\theta_{k}(gu(x)g^{-1})W(g)$, which is equal $\theta_{k}(zku(x)(cz)^{-1})W(g)$ because $P_{k-1}$ normalizes $\theta_{k}$.
 This implies the equality $[\theta_{k}(zku(x)(kz)^{-1})-1]W(g)=[\theta_{k}(u(zkx))-1]W(g)=0$ for 
any $x$ in a neighbourhood of zero depending only on $W$. 
The assertion follows easily. 
\end{rem}

\begin{prop}\label{Ker}
Let $(\tau,V)$ be a $P_k$-submodule of $C^{\infty}(N_k\backslash P_k,\theta^k)$. Then the space $V(U_k,1)$ is the subspace of $V$, of functions $W$ 
such that there exists an integer $N_W$ with $W(g)=0$, for any $g$ satisfying $|z(g)|_F\leq q_F^{-N_W}$. 
\end{prop}
\begin{proof}
 Suppose first that a function $W$ is in $V(U_k,1)$, so we can write it $\pi(u)W'-W'$ for some $u$ in $U_k^{ab}$ and some $W'$ in $V$.
 Then, writing $g$ as $pzk$, and $u$ as $u(x)$ for $x$ in $F^{m}$, we have $[\pi(u)W'-W'](g)= [\theta_k(u(zkx))-1]W'(g)$, 
which will be zero to $0$ when $|z|_F$ is close to zero.\\
Conversely, we use the characterization of Jacquet and Langlands asserting that the elements $W$ of $V(U_k,1)=V(U_k^{ab},1)$ are those such that $\int_{U}\tau(u)Wdu$ is zero as
 soon as the open compact subgroup $U$ of $U_k^{ab}$ contains some compact open subgroup $U_W$. So suppose $W$ is in $V$ and that it vanishes on elements
$g$ of $G_{n-1}(F)$ satisfying $|z(g)|_F\leq q_F^{-N_W}$.\\
 Let $U$ be any open compact subgroup of $U_k^{ab}$, 
that we identify with a subgroup of $F^{m}$. The integral $\int_{U}\tau(u)Wdu$ evaluated at $g=pzk$, is equal to 
$\int_{x\in U}\theta_k(zkx)W(g)dx$. Hence this integral is always zero for $|z|_F\leq q_F^{-N_W}$ because $W(g)$ is.\\
 We now recall that as $\theta_k$ is a non trivial character of $U_k^{ab}$, there exists a compact open ball $U_0$ of $U_k^{ab}\simeq F^n$ such that, 
the integral $\int_{x\in U}\theta_k(x)dx$ is zero whenever the compact open subgroup $U$ of $U_k^{ab}$ contains $U_0$. But then for $|z|_F\geq q_F^{-N_W}$, 
if $t_W$ is an element of $F^*$ of absolute value $q_F^{N_W}$, the integral $\int_{x\in U}\theta_k(zkx)W(g)dx$ is also zero as soon as $U$ contains $t_W U_0$. 
Hence $\int_{U}\tau(u)Wdu$ is zero when $U$ is a compact open subgroup of 
$U_k^{ab}$ containing $U_W=t_W U_0$, and $W$ belongs to $V(U_k,1)$. 

\end{proof}

For any smooth $P_n$-module $\tau$, and any integer $k\geq1$, we denote by $\tau_{(k)}$ the representation of $P_{n-k+1}$ equal to $\Phi^{k-1} \tau$, and by
$\tau^{(k)}$ the representation of $G_{n-k}$ equal to $\Psi^-\Phi^{k-1} \tau=\Psi^-\tau_{(k)}$.\\
We say that a smooth irreducible representation $\pi$ of $G_n$ is $\theta^n$-generic if it is isomorphic to a submodule of the induced representation 
$Ind_{N_n}^{G_n}(\pi)$. If it is the case, the submodule of $Ind_{N_n}^{G_n}(\pi)$ isomorphic to $\pi$ is unique, it is called the Whittaker
 model of $\pi$ and denoted by $W(\pi,\theta^n)$.\\
Now we let $(\pi,V)$ be a $\theta^n$-generic representation of $G_n$ (hence a smooth $P_n$-module as well), we denote by $(\pi',V')$ the representation of $P_n$ 
obtained on the space of restrictions of functions in $W(\pi,\theta)$ to $P_n$, it is a quotient of $\pi$ as a $P_n$-module, and restriction to $P_n$ is known to be 
an isomorphism in case A.\\ 

The following proposition follows from applying repeatedly Proposition \ref{stab}, and from Proposition \ref{Ker}.

\begin{prop}\label{Ker2}
Let $\tau$ be a smooth $P_n$-submodule of $C^{\infty}(N_n\backslash P_n,\theta^n)$, and $k\geq 0$ be an integer, then the $P_{k+1}$-module 
$\tau_{(n-k-1)}$ is a submodule of $C^{\infty}(N_{k+1}\backslash P_{k+1},\theta^{k+1})$,
 with model given by restriction of functions $[\delta_{U_{k+2}}\dots\delta_{U_n}]^{-1/2}W$ in $\tau$ to $P_{k+1}$.
  In this realisation, one has 
$\tau_{(n-k-1)}(p)W= \rho (p)W$ for $p$ in $P_{k+1}$, where $\rho$ is the action by right translation.\\ 

\end{prop}

The next proposition asserts amongst other things 
that for every $k\geq1$, 
the $G_{n-k}$-module $\pi^{(k)}$ has finite length. 

\begin{prop}\label{factorseries}
 If $(\pi,V)$ is a smooth representation of $G_n$ of finite length, then for $k$ between $1$ and $n-1$, the $G_k$-module $\pi^{(n-k)}$ 
(hence its quotient $\pi'^{(n-k)}$) has finite length.
 
\end{prop}
\begin{proof} 
 For $k\geq1$, except in case D, $k=2$, we denote by $U_{k,n-k}$ the unique standard unipotent radical (denoted by $U_{(k;n-k)}$ in the previous section) containing $U_{\alpha_{k}}$ as only simple root subgroup.\\
In case $D$, for $k=2$, we be denote $U_{2,n-2}$ the unique standard unipotent radical containing $U_{\alpha_{1}}$ and $U_{\alpha_{2}}$ as only simple root subgroups.\\
In all cases, the corresponding Levi $M_{k,n-k}$ is the direct product of $G_k$ with $GL(n-k,F)$.\\
 Now the module $G_k$-module $\pi^{(n-k)}$ is a quotient of the Jacquet $G_k\times GL(n-k)$-module $$(\pi_{U_{k,n-k}},V/V(U_{k,n-k},1)),$$ as the kernel of 
the surjective map $\pi \twoheadrightarrow \pi^{(n-k)}$ contains $V(U_{k,n-k},1)$. 
More precisely, let $N_{n-k,A}$ be the unipotent radical of the standard Borel subgroup of $GL(n-k,F)$, the group $U_{k+1}\dots U_n$ is the 
semidirect product $N_{n-k,A}\ltimes U_{k,n-k}$, so that the space $V^{(n-k)}$ of $\pi^{(n-k)}$ is equal to 
the quotient $$V/V(N_{n-k,A}\ltimes U_{k,n-k},\theta^n_{|N_{n-k,A}}\otimes 1_{U_{k,n-k}})$$ where $V$ is the space of $\pi$.\\
 We denote by 
$I_{k}$ the surjection obtained by facorisation from $V_{U_{k,n-k}}$
onto $V^{(n-k)}$. From Lemma 2.32 of \cite{BZ}, the map $I_{k}$ identifies with the projection 
$$V_{U_{k,n-k}} \twoheadrightarrow (V_{U_{k,n-k}})_{N_{n-k,A},\theta^n_{|N_{n-k,A}}}=V_{U_{k,n-k}}/V_{U_{k,n-k}}(N_{n-k,A},\theta^n_{|N_{n-k,A}}).$$ 
The map $I_k$ is in fact a $G_k$-modules morphism, because of the equality of modulus characters 
$$(\delta_{U_{k,n-k}})_{|G_k}=(\delta_{U_{k+1}}\dots \delta_{U_{n}})_{|G_k}$$ which is itself a consequence of the decomposition 
$$U_{k,n-k}=\prod_{i=k+1}^n (U_{k,n-k}\cap U_i).$$
 
The group $N_{A,n-k}$ being a union of compact subgroups, 
the map $I_k$ preserves exact sequences.\\
As the Jacquet module functor preserves finite length, the $G_k\times GL(n-k,F)$-module $\pi_{U_{k,n-k}}$ has a finite composition series 
$0\subset (\pi_{U_{k,n-k}})_1\subset\dots\subset (\pi_{U_{k,n-k}})_{r_k}=\pi_{U_{k,n-k}}$. We put $\pi^{(n-k)}_i=I_k[(\pi_{U_{k,n-k}})_i]$.\\
 Hence $\pi^{(n-k)}_i/\pi^{(n-k)}_{i-1}$ is equal to $[(\pi_{U_{k,n-k}})_i/(\pi_{U_{k,n-k}})_{i-1}]_{N_{n-k,A},\theta_{|N_{n-k,A}}}$, but as a 
 $G_k\times GL(n-k,F)$-module, the quotient $(\pi_{U_{k,n-k}})_i/(\pi_{U_{k,n-k}})_{i-1}$ isomorphic to $\rho_1\otimes\rho_2$ for irreducible representations 
$\rho_1$ and $\rho_2$ of $G_k$ and $GL(n-k,F)$ respectively. Because the character $\theta^n$ restricts to $N_{n-k,A}$ as a nondegenerate character, the quotient 
$\pi^{(n-k)}_i/\pi^{(n-k)}_{i-1}$ is equal to 
$\rho_1\otimes (\rho_2)_{N_{n-k,A},\theta^n_{|N_{n-k,A}}}$, thus it is zero unless $\rho_2$ is generic, in which case it is equal to the irreducible 
representation $\rho_1$.\\
So we proved that $\pi^{(n-k)}$ has finite length as $G_k$-module, smaller than the length of the Jacquet module 
$\pi_{U_{k,n-k}}$ as a $G_k\times GL(n-k,F)$-module.  
 
\end{proof}

There is another property of the maps $I_k$ defined in the proof of the preceding proposition that is worth mentioning, which is 
that their restriction to generalised characteristic subspaces is nonzero. More formally, let $G$ be an $l$-group, 
and $T$ be a closed abelian subgroup of $G$. If $V$ is a smooth $G$-module, following \cite{C}, we define for each character $\chi$ of 
$T$, the $T$-submodule 
$$V_{\chi,\infty}=\{v\in V|\  \exists n\in \mathbb{N}, \forall t\in T, \ (\tau(t)-\chi(t)Id)^n(v)=0\}.$$ 
If $V$ is $T$-finite (i.e. every vector in $V$ generates a finite dimensional $T$-module), then it is the finite direct sum of its 
(nonzero by definition) generalised characteristic subspaces, 
and every such (nonzero) $V_{\chi,\infty}$ contains the nonzero generalised eigenspace 
$$V_{\chi}=\{v\in V|\  \exists n\in \mathbb{N}, \forall t\in T, \ (\tau(t)-\chi(t)Id)(v)=0\}.$$

First recall that smooth $(F^*)^r$-modules $E$, with a filtration 
${0}=E_0 \subset E_1 \subset \dots \subset E_{r-1} \subset E_{r}=E$ such that $(F^*)^r$ acts by a character on each quotient are 
$(F^*)^r$-finite.

\begin{LM}\label{fin-dim}
 Let $E$ be a smooth $(F^*)^r$-module $E$, with a filtration 
${0}=E_0 \subset E_1 \subset \dots \subset E_{r-1} \subset E_{r}=E$ such that $(F^*)^r$ acts by a character $c_{i+1}$ on each quotient $E_{i+1}/E_i$, then any vector of $E$ lies in a finite dimensional $(F^*)^r$-submodule.
\end{LM}

\begin{proof}
One proves this by induction on the smallest $i$ such that $E_i$ contains $v$. If this $i$ is $1$, the group $(F^*)^r$ only multiplies $v$ by a scalar, and we are done.\\
Suppose that the result is known for $E_i$, and take $v$ in $E_{i+1}$ but not in $E_i$. Then for every $t$ in $(F)^*$, the vector $\tau(t)v-c_{i+1}(t)v$ belongs to $E_i$. 
By smoothness, the set $\{ \tau(u)v \ | t\in (U_F)^r \}$ is actually equal to $\{ \tau(u)v \ | u\in P \}$ for $P$ a finite set of $(U_F)^r$. The vector space generated by this set
is stabilized by $(U_F)^r$, and has a finite basis $v_1,\dots,v_m$. Now the vectors $$\tau(1,\dots,1,\w,1,\dots,1)v_l-c_{i+1}(1,\dots,1,\w,1,\dots,1)v_l$$
 belong to $E_i$, hence by induction hypothesis, to a finite dimensional 
$(F^*)^r$-submodule $V_l$ of $E_i$.
Finally the finite dimensional space $Vect(v_1,\dots,v_m)+V_1+\dots+V_m$ is stable under $(U_F)^r$ and the elements $(1,\dots,1,\w,1,\dots,1)$,
 hence $(F^*)^r$, and contains $v$.
\end{proof}

This in particular applies to the $Z_kZ_n$-module $V_{U_{k,n-k}}$ and the $Z_k$-module $V$ described in the proof of Proposition \ref{factorseries}, as both are respectively 
$G_k\times GL(n-k,F)$ and $G_k$-modules of finite length.\\ 

 Now we can prove the following property of the maps $I_k$:

\begin{prop}\label{weights}
Let $(\pi,V)$ be a $\theta^n$-generic representation of $G_n$, and for $k\geq1$, let $U_{k,n-k}$ and $M_{k,n-k}\simeq G_k\times GL(n-k,F)$ the subgroups of $G_n$ defined in the proof of 
Proposition \ref{factorseries}. 
Let $\chi$ be a character of the central subgroup $Z_kZ_n$ of $M_{k,n-k}$, and denote by the same letter its restriction to the central 
subgroup $Z_k$ of $G_k$. If the 
generalised characteristic subspace $(V_{U_{k,n-k}})_{\chi,\infty}$ is nonzero, then the map $I_k$ restricts non trivially to $(V_{U_{k,n-k}})_{\chi}$. 
In particular the space $V^{(n-k)}_{\chi}$ is nonzero.  
\end{prop}
\begin{proof}
 Suppose that the subspace $(V_{U_{k,n-k}})_{\chi,\infty}$ of $V_{U_{k,n-k}}$ is nonzero, hence $(V_{U_{k,n-k}})_{\chi}$ is nonzero. The space $(V_{U_{k,n-k}})_{\chi}$ is $M_{k,n-k}$-submodule of 
 $V_{U_{k,n-k}}$, so it has finite length, hence it contains some irreducible representation $\rho_1\otimes \rho_2$ of $M_{k,n-k}$. 
Hence $Hom_{M_{k,n-k}}(\rho_1\otimes \rho_2,V_{U_{k,n-k}})$ is nonzero, but then from Bernstein's second adjointness theorem (see \cite{Bu}, Theorem 3), we deduce that
$V$ is a quotient of the representation $\rho_1\times \rho_2$ parabolically induced from $\rho_1\otimes \rho_2$. As $V$ admits a nonzero Whittaker form, so does 
$\rho_1\times \rho_2$, and from a classical result of Rodier (Theorem 7 of \cite{R}), this implies that $\rho_1$ and $\rho_2$ are generic with respect to some 
nondegenrate character. As generiticity doesn't depend on the character for $GL(n-k,F)$, we deuce that $I_k(\rho_1\otimes \rho_2)=\rho_1$. Hence $I_k$ 
restricts non trivially to $(V_{U_{k,n-k}})_{\chi}$, and the image $I_k[(V_{U_{k,n-k}})_{\chi}]$ contains $\rho_1$ which is a nonzero submodule of $V^{(n-k)}_{\chi}$.
\end{proof}

We will also need to know that, if $k$ is an integer between $1$ and $n-1$ and $\chi$ is a character of $Z_k$, then the $Z_k$-modules $(V^{(n-k)})_{\chi}$ and 
$(V'^{(n-k)})_{\chi}$ are nonzero at the same time. We already know from the previous proposition that this is equivalent to the fact that the
 $Z_k$-modules $(V'^{(n-k)})_{\chi}$ and $(V_{U_{k,n-k}})_{\chi}$ are nonzero at the same time, and that $(V'^{(n-k)})_{\chi}$ nonzero implies 
that $(V_{U_{k,n-k}})_{\chi}$ is nonzero.

\begin{prop}\label{weights2}
If $(\pi,V)$ is a $\theta^n$-generic representation of $G_n$, and for $k\geq1$, let $U_{k,n-k}$ and $M_{k,n-k}\simeq G_k\times GL(n-k,F)$ be the subgroups of $G_n$ defined in the proof of 
Proposition \ref{factorseries}. 
Let $\chi$ be a character of the central subgroup $Z_kZ_n$ of $M_{k,n-k}$, and denote by the same letter its restriction to the central subgroup 
$Z_k$ of $G_k$, then the space
 $(V_{U_{k,n-k}})_{\chi}$ is nonzero if and only if the space $(V'^{(n-k)})_{\chi}$ is nonzero.  
\end{prop}
\begin{proof}
We only need to prove that if $(V_{U_{k,n-k}})_{\chi}$ is nonzero, then the space $(V'^{(n-k)})_{\chi}$ is nonzero.\\
So suppose that the $M_{k,n-k}$-module $(V_{U_{k,n-k}})_{\chi}$ is nonzero, it is of finite length, hence it contains an irreducible 
$M_{k,n-k}$-submodule $\rho$. Call $P_{k,n-k}$ the parabolic subgroup $M_{k,n-k}U_{k,n-k}$, and $P_{k,n-k}^-$ its opposite parabolic subgroup 
(with unipotent radical $(U_{k,n-k})^-$).
 We already saw that by Bernstein's second adjointness theorem, the induced representation 
$i_{P_{k,n-k}^-}^{G_n}(\rho)$ has $\pi$ as a quotient, and therefore $i_{P_{k,n-k}^-}^{G_n}(\rho)$ is $\theta^n$-generic. Then from 
Theorem 7 of \cite{R}, the $M_{k,n-k}$-module $\rho$ is 
$\theta$-generic, where $\theta$ is the restriction of $\theta^n$ to the unipotent radical of the Borel of $M_{k,n-k}$.
 Both have the same Whittaker model $W(\pi,\theta^n)$, i.e. the (unique up to scalar) Whittaker form on the space of $i_{P_{k,n-k}^-}^{G_n}(\rho)$ facorises through the 
projection from $i_{P_{k,n-k}^-}^{G_n}(\rho)$ to $\pi$. Let $L^-$ be a nonzero $\theta$-Whittaker form on the space of $\rho$, 
by Theorems 1.4 and 1.6 of \cite{CS}, there is a nonzero Whittaker form $L$ on the space of $i_{P_{k,n-k}^-}^{G_n}(\rho)$
 whose restriction to the subspace $C_c^\infty(P_{k,n-k}^-\backslash P_{k,n-k}^-U_{k,n-k},(\delta_{U_k^-})^{1/2}\rho)$ of functions with support in $P_{k,n-k}^- U_{k,n-k}$ is given by 
 $$f\mapsto \int_{U_{k,n-k}} L^-(f(u))\theta^{-1}(u)du.$$
We denote by $\bar{L}$ the Whittaker form on the space of $\pi$ which lifts to $L$.
In particular, for any $\bar{f}$ in the space of $\pi$, which is the image of $f$ in $C_c^\infty(P_{k,n-k}^-\backslash P_{k,n-k}^-U_{k,n-k})$, one has 
$$\bar{L}(\bar{f})=\int_{U_{k,n-k}} L^-(f(u))\theta^{-1}(u)du.$$
Let $v$ be a vector in the space of $\rho$, such that $L^-(v)$ is nonzero. Let $K'$ be a compact subgroup of 
$G_n$, with Iwahori decomposition with respect to $P_{k,n-k}$, and such that 
$K'\cap M_{k,n-k}$ fixes $v$, then the function $\alpha$ equal to $u^-m u \mapsto \rho(m)v$ on $(U_{k,n-k})^- M_{k,n-k} (U_{k,n-k}\cap K')$, and zero outside,
 is well defined and belongs to the space $C_c^\infty(P_{k,n-k}^-\backslash P_{k,n-k}^-U_{k,n-k},\rho)$. We denote by $W_\alpha$ the corresponding
  Whittaker function $g\mapsto \bar{L}(\pi(g)\bar{\alpha})$. If $a$ belongs to the group $Z_k$, one has 
 $$\begin{aligned} W_\alpha(a) & = \int_{U_{k,n-k}} L^-(\alpha(ua))\theta^{-1}(u)du \\
                                 & = \chi(a)\delta_{U_{k,n-k}}^{-1/2}(a)\int_{U_{k,n-k}} L^-(\alpha(a^{-1}ua))\theta^{-1}(u)du \\
                                 & = \chi(a)\delta_{U_{k,n-k}}^{1/2}(a)\int_{U_{k,n-k}} L^-(\alpha(u))\theta^{-1}(aua^{-1})du \\
                                 & = \chi(a)\delta_{U_{k,n-k}}^{1/2}(a)L^-(v) \int_{U_{k,n-k}}\theta^{-1}(aua^{-1})du \end{aligned}$$
                                 
In this last integral, $u$ stays in the compact set $U_{k,n-k}\cap K'$, hence there is a (punctured) neighbourhood of zero in $Lie(Z_k)=F$, such that 
$a(U_{k,n-k}\cap K')a^{-1}$ is a subset of $Ker(\theta)$ when $a$ belongs to this neighbourhood. Finally, up to multiplication of the function 
$\alpha$ by a scalar, one has
 $$W_{\alpha}(a)=\chi(a)\delta_{U_{k,n-k}}^{1/2}(a)L^-(v)$$ whenever $a$ is this neighbourhood of zero.\\
 A similar computation gives the equality $$W_{\alpha}(zg)=\chi(z)\delta_{U_{k,n-k}}(z)^{1/2}c_{g}$$ for $z$ in $Z_k$ in a neighbourhood of zero and $g$ in $G_k$,
 where  $c_{g}$ is the constant $\int_{U_{k,n-k}} L^-(\alpha(ug))du$.                          
Hence from Propositions \ref{Ker}, \ref{Ker2}, and the equality $(\delta_{U_{k,n-k}})_{|G_k}=(\delta_{U_{k+1}}\dots \delta_{U_{n}})_{|G_k}$, we deduce that 
the vector $(\delta_{U_{k+2}}\dots \delta_{U_{n}})^{-1/2} W_{\alpha}$ in the space of $\pi'_{(n-k+1)}$ is such that its image in $\pi'^{(n-k)}$ is nonzero and 
belongs to the space $(\pi'^{(n-k)})_{\chi}$. This proves the proposition.
\end{proof}

A straightforward generalisation of the proof of the preceding proposition gives the following corollary.

\begin{cor}
Let $G$ be the $F$-points of a quasi-split reductive group defined over $F$. Let $P$ be a parabolic subgroup of $G$ with a Levi subgroup $M$, and $P^-$ 
its opposite subgroup with $P\cap P^-=M$. 
Let $(\pi,V)$ be a smooth $\theta$-generic representation of $G$, for some nondegenerate character $\theta$ of the unipotent radical $U$ of a Borel 
subgroup of 
$G$ contained in $P$. We denote by $j_{P^-}$ be the map defined in Theorem 3.4 of \cite{D} from $(V^*)^{U_{\theta}}$ to $(J_P(V)^*)^{M\cap U_{\theta}}$.
 If $L$ is a nonzero vector of the line $(V^*)^{U_{\theta}}$, then the linear form $j_{P^-}(L)$ restricts non trivially to any 
irreducible $M$-submodule of $J_P(V)$ whenever the Jacquet module $J_P(V)$ is nonzero.
\end{cor}

We now come back to the study of $F^*$-modules $E$ with finite factor series such that $F^*$ acts by a character on each quotient. From Lemma \ref{fin-dim}, 
any vector of $E$ will belong to a finite dimensional $F^*$-submodule $E'$ as in:

\begin{prop}\label{center-repr}
If $E'$ is a non zero finite dimensional $F^*$-submodule of $E$, then $E'$ has a basis $B$ in which the action of $F^*r$ is given by a block diagonal matrix 
$Mat_B (\tau(t))$ with each block of the form:
$$ \begin{pmatrix}
c (t) & c (t)P_{1,2} (v_F(t))     & c (t)P_{1,3} (v_F(t))     & \dots                      & c (t)P_{1,q} (v_F(t))   \\
                    & c (t) & c (t)P_{2,3} (v_F(t))     &  \dots                     & c (t)P_{2,q} (v_F(t))   \\
                    &                    & \ddots              &                            & \vdots            \\  
                    &                    &                     & c (t)      &  c (t)P_{q-1,q} (v_F(t)) \\ 
                    &                    &                     &                            & c(t)                                                                                                     
  \end{pmatrix},$$
for $c$ one of the $c_i$'s, $q$ a positive integer depending on the block, and the $P_{i,j}$'s being polynomials with no constant term of degree at most $j-i$.

\end{prop}

\begin{proof}
First we decompose $E'$ as a direct sum under the action of the compact abelian group $U_F$. Because $E'$ has a filtration by the spaces $E' \cap E_i$,
and that $F^*$ acts on each sub factor as one of the $c_i$'s, the group $U_F$ acts on each weight space as the restriction of one of the $c_i$'s. Now each weight space is
stable under $F^*$ by commutativity, and so we can restrict ourselves to the case where $E'$ is a weight-space of $U_F$.\\
Again $E'$ has a filtration, such that $F^*$ acts on each sub factor as one of the $c_i$'s (with all these characters having the same restriction to $U_F$), 
say $c_{i_1}, \dots, c_{i_k}$, in particular, we deduce that the endomorphism $\tau(\w)$ has a triangular matrix in a basis adapted to this filtration, 
with eigenvalues $c_{i_1}(\w), \dots, c_{i_k}(\w)$. As $\tau(\w)$ is trigonalisable, the space $E'$ is the direct sum its characteristic subspaces, and
again these characteristic subspaces are stable under $F^*$.\\
 So finally one can assume that $E'$ is a characteristic subspace for some eigenvalue $c(\pi)$ of $\tau(\w)$, on which $U_F$ acts as the character $c$, 
where $c$ is one of the $c_i$'s. \\
 
Hence there is a basis $B$ of $E'$ such that $$Mat_B (c^{-1}(t)\tau(t))= \begin{pmatrix}
1 & A_{1,2} (t)     & A_{1,3}(t)     & \dots                      & A_{1,q}  (t)   \\
                    & 1 & A_{2,3} (t)     &  \dots                     & A_{2,q} ( t )   \\
                    &                    & \ddots              &                            & \vdots            \\  
                    &                    &                     & 1      &  A_{q-1,q} ( t ) \\ 
                    &                    &                     &                            & 1                                                                                                    
  \end{pmatrix} $$ for any $t$ in $F^*$, where the $A_{i,j}$'s are smooth functions on $F^*$. So we only have to prove that the $A_{i,j}$'s are polynomials
 of the valuation of $F$ with no constant term.\\
We do this by induction on $q$.\\
 It is obvious when $q=1$.
Suppose the statement holds for $q-1$, and suppose that $E'$ is of dimension $q$, with basis $B=(v_1,\dots,v_{q})$. Considering the two 
$c^{-1}\tau(F^*)$-modules $Vect(v_1,\dots,v_{q-1})$ and  $Vect(v_1,\dots,v_{q})/Vect(v_1)$ of dimension $q-1$, we deduce that for every couple $(i,j)$ different from 
$(1,q)$, there is a polynomial with no constant term $P_{i,j}$ of degree at most $j-i$, such that $A_{i,j}=P_{i,j}\circ v_F$. Now because $c^{-1}\tau$ 
is a representation of $F^*$, and because the $P_{i,j}\circ v_F$'s vanish on $U_F$ for $(i,j)\neq (1,q)$, we deduce that $A_{1,q}$ is a smooth morphism from $(U_F,\times)$ 
to $(\mathbb{C},+)$, which must be zero because $(\mathbb{C},+)$ has no nontrivial compact subgroups. From this we deduce that $A_{1,q}$ is invariant under translation by
 elements of $U_F$ (i.e. $A_{1,q}(\w^k u)= A_{1,q}(\w^k )$ for every $U$ in $U_F$).\\
Denote by $M(k)$ the matrix $Mat_B (c^{-1}\tau(\w^k))$ for $k$ in $\mathbb{Z}$. One has $M(k)=M(1)M(k-1)$ for $k\geq 1$, which in implies 
$A_{1,q}(\w^{k})= \sum_{j=2}^{q-1} P_{1,j}(1)P_{j,q}(k-1) + A_{1,q}(\w^{k-1})+ A_{1,q}(\w)= Q(k)+ A_{1,q}(\w^{k-1})+ A_{1,q}(\w)$ for $Q$
 a polynomial of degree at most $q-2$. This in turn implies that $A_{1,q}(\w^{k})= \sum_{l=1}^{k-1} Q(l) +k A_{1,q}(\w)= R(k)$ for $R$ a polynomial of degree at most $q-1$,
 according to the theory of Bernoulli polynomials, for any $k\geq0$. The same reasoning for $k\leq0$, implies $A_{1,q}(\w^{k})=R'(k)$ for $R'$ a polynomial 
of degree at most $q-1$, for any $k\leq 0$. We need to show that $R=R'$ to conclude.\\ 
We know that $M(k)$ is a matrix whose coefficients are polynomials in $k$ for $k>0$ of degree at most $q-1$, we denote it by $P(k)$. The matrix $M(k)$
 has the same property for $k<0$, 
we denote it by $P'(k)$.
 Moreover for any $k\geq0$ and $k'\leq 0$, with $k+k'\geq 0$, one has $P(k+k')=P(k)P'(k')$. Fix $k >q-1$, then the matrices $P(k +k')$ and $P(k )P'(k')$ 
are equal for $k'$ in $[1-q,0]$, as their coefficients are polynomials in $k'$ with degree at most $q-1$, the equality $P(k +z')=P(k )P'(z')$ holds for any complex 
number $z'$. Now fix such a complex number $z' $, the equality $P(k +z')$ and $P(k )P'(z')$ holds for any integer $k >q-1$, and as both matrices have coefficients which
 are polynomials in $k$, this equality actually holds for any complex number $z$, so that $P(z +z')$ equals $P(z )P'(z')$ for any complex numbers $z$ and $z'$.\\
 As $P(0)=I_q$, we deduce that $P$ and $P'$ are equal on $\mathbb{C}$, and this implies that $R$ is equal to $R'$. 
\end{proof}

From this we deduce the following theorem, giving an expansion at infinity of Whittaker functions of generic representations of $G_n$. 
We consider $\pi$ a $\theta$-generic representation of $G_n$. Let $R$ be the set of positive integers between $1$ and $n-1$ such that the nonzero derivatives of $\pi$ with respect to $\theta$ are the finite length representations $\pi^{(n-k)}$ of $G_k$, for $k\in R$. We then denote for each $k\in R$, by $(c_{l_k,k})_{l_k}$ the family of characters of $Z_k$ given by the action of this group on the irreducible sub-quotients of $\pi^{(n-k)}$. We finally set $V=\{1,\dots,n-1\}-R$.

\begin{thm}\label{restorus}
For any $W$ in $W(\pi,\theta)$, the function \[W(z_1 \dots z_{n-1})\] is a linear combination of products of the form 
\[\prod_{k\in R} [c_{l_k,k}\d_{U_k+1}^{1/2}\dots \d_{U_n}^{1/2}](z_k)v^{m_k}(z_k)\phi_k(z_k) \prod_{l\in V}\phi_l(z_l),\] for non negative integers $m_k$, functions $\phi_k$ in $\mathcal{C}_c^\infty(Lie(Z_k))$, and 
functions $\phi_l$ in $\mathcal{C}_c^\infty(Lie(Z_l))$ vanishing at $z_l=0$.
\end{thm}

\begin{proof} Actually we prove the following stronger statement, which is satisfied by $\pi_{(0)}$ according to Proposition \ref{factorseries}: 
\begin{st}
Let $\pi$ be a finite length submodule of 
$\mathcal{C}^\infty(N_n\backslash P_n,\theta)$. Let $R$ be the set of positive integers between $1$ and $n-1$ such that the nonzero derivatives of $\pi$ with respect to $\theta$ are the finite length representations $\pi^{(n-k)}$ of $G_k$, for $k\in R$. We then denote for each $k\in R$, by $(c_{l_k,k})_{l_k}$ the family of characters of $Z_k$ given by the action of this group on the irreducible sub-quotients of $\pi^{(n-k)}$. We finally set 
$V=\{1\dots,n-1\}-R$. For any $W$ in $\pi$, the function \[W(z_1,\dots,z_{n-1})=W(z_1 \dots z_{n-1})\] is a linear combination of products of the form \[\prod_{k\in R} [c_{l_k,k}\d_{U_k+1}^{1/2}\dots \d_{U_n}^{1/2}](z_k)v^{m_k}(z_k)\phi_k(z_k) \prod_{l\in V}\phi_l(z_l),\] for non negative integers $m_k$, functions $\phi_k$ in $\mathcal{C}_c^\infty(Lie(Z_k))$, and 
functions $\phi_l$ in $\mathcal{C}_c^\infty(Lie(Z_l))$ vanishing at $z_l=0$.
\end{st}
 
 The proof is by induction on $n$.\\
 Let $W$ belong to the space of $\pi$. We denote by $v$ its image in the space $E$ of $\pi^{(1)}$. 
The vector $v$ belongs to a finite dimensional $Z_{n-1}$-submodule $E'$ of $E$, on which $Z_{n-1}$ acts by a matrix of 
the form determined in Proposition \ref{center-repr}. If $\pi^{(1)}=\{0\}$, we set $v=0$.
If $\pi^{(1)}\neq\{0\}$, we fix a basis $B=(e_{1},\dots,e_{q})$ of $E'$, and denote by $M(a)$ the matrix $M_B(\tau(a))$ (with $a$ in $Z_{n-1}$ and 
$\tau(a)=\pi^{(1)}(a)$), hence we have $\tau(a)e_l= \sum_{k=1}^q M(a)_{k,l}e_k$ for each $l$ between $1$ and $q$.\\
Taking preimages $\tilde{E}_{1},\dots,\tilde{E}_{q}$ of $e_{1},\dots,e_{q}$ in 
$\pi_{(0)}$, we denote by $\tilde{E}$ the function vector $\begin{pmatrix}
\tilde{E}_{1} \\ \vdots \\  \tilde{E}_{q} \end{pmatrix}$.\\   
If the image $v$ of $W$ in $\pi^{(1)}$ is equal to $x_1 e_1+\dots+x_q e_q$, there is an integer $M$, 
such that for every $(z_1,\dots,z_{n-2})$ in $Z_1\times\dots\times Z_{n-2}$, the function 
$$W(z_1,\dots,z_{n-1})-(x_1,\dots,x_q)\tilde{E}(z_1,\dots,z_{n-1})$$ vanishes for $|z_{n-1}|_F\leq q_F^{-M}$. We denote by $S$ the function 
$(x_1,\dots,x_q)\tilde{E}$. If $\pi^{(1)}=\{0\}$ we set 
$S=0$.\\ 
Because of Remark \ref{vanishlarge}, there is an integer $M'$, such that for any $(z_1,\dots,z_{n-2})$ in 
$Z_1\times\dots\times Z_{n-2}$, and any $z_{n-1}$ in $Z_{n-1}$ of absolute value greater than $q_F^{M'}$, 
both $W(z_1,\dots,z_{n-1})$ and $S(z_1,\dots,z_{n-1})$ are zero, so that the difference $D(z_1,\dots,z_{n-1})$ of the two
functions is a smooth function which vanishes whenever $z_{n-1}$ has absolute value outside $[q_F^{-M},q_F^{M'}]$.
 Moreover there is a compact subgroup $U$ of $Z_{n-1}(\o)$ independent of $(z_1,\dots,z_{n-1})$ such that both functions (hence $D$) are invariant when $z_{n-1}$ 
is multiplied by an element of $U$. Denoting by $(z_\alpha)_{\alpha \ \in \ A}$ a finite set of representatives for $$\{ z\ | q_F^{-M}\leq |z_{n-1}|_F \leq q_F^{M'}\}/U,$$
 this implies that $D(z_1,\dots,z_{n-1})$ is equal to $\sum_{\alpha \ \in \ A} D(z_1,\dots,z_{n-2},z_{\alpha})\mathbf{1}_{z_\alpha U}(z_{n-1})$,
 which we can always write as $\sum_{\alpha \ \in \ A} D(z_1,\dots,z_{n-2},z_{\alpha})\d_{U_n}^{1/2}(z_{n-1}) D_\alpha(z_{n-1})$ with 
$D_\alpha=\d_{U_n}^{-1/2}\mathbf{1}_{z_\alpha U}$ in $C_c^\infty(Lie(Z_{n-1}))$.\\ 
Each function $D(z_1,\dots,z_{n-2},z_{\alpha})$ is equal to $W(z_1,\dots,z_{\alpha})-S(z_1,\dots,z_{\alpha})$, and the restrictions to $P_{n-1}$ of the 
functions $\delta_{U_n}^{-1/2}[\pi(z_{\alpha}) D]$ belong to the smooth submodule $\Phi^-(\pi)$ of 
$C^\infty(N_{n-1}\backslash P_{n-1},\theta)$, which still satisfies the hypothesis of the statement.\\
Hence, by induction hypothesis, the function 
$D$ is of the expected form. Notice that if 
$\pi^{(1)}=\{0\}$, one has $W=D$ and this ends the proof.\\
                                                                                                                               
Now, call $p$ the projection $W'\mapsto (\d_{U_n}^{1/2} W')_{|P_{n-1}}$ from $\pi_{(0)}$ to $\pi^{(1)}$. 
If $\pi^{(1)}\neq \{0\}$, then for any $a$ in $Z_{n-1}$, one has $\rho(a)p(\tilde{E}_{l})= \sum_{k=1}^q M(a)_{k,l}p(\tilde{E}_{k})$. Hence as $\rho(a)p(\tilde{E}_{l})$ equals $\d_{U_n}^{-1/2}(a)\pi_{(0)}(a)\tilde{E}_{l}$,
 we deduce that 
there is a punctured neighbourhood of zero in $Z_{n-1}$, such that for each $l$, the function 
$\d_{U_n}^{-1/2}(a)\pi_{(0)}(a)\tilde{E}_{l} -\sum_{k=1}^q M(a)_{k,l}\tilde{E}_{k}$ 
 vanishes on elements $g=pac$ of $G_{n-1}$ ($p$ in $P_{n-1}$, $a$ in $Z_{n-1}$, $c$ in $G_{n-1}(\o)$) such that $a$ is in this neighbourhood.\\
In particular, there exists $N_a$ such that for every $(z_1,\dots,z_{n-1})$, the vector function 
$$\d_{U_n}^{-1/2}(a)\pi_{(0)}(a)\tilde{E}(z_1,\dots,z_{n-1}) -^t\!M(a)\tilde{E}(z_1,\dots,z_{n-1})$$ 
vanishes when we have $|z_{n-1}|_F\leq q_F^{-N_a}$.\\

This implies, as in the proof of Proposition 2.6. of \cite{CP}, the following claim:

\begin{claim}\label{claim}
There is actually an $M''$, such that for every $z$ in $Z_{n-1}$, with $|z_{n-1}|_F\leq q_F^{-M''}$, and every $a$ in $Z_{n-1}$, 
with $|a|_F\leq 1$, the function $\tilde{E}(z_1,\dots,z_{n-1}a)$ is equal to $\d_{U_n}^{1/2}(a){^t\!M(a)}\tilde{E}((z_1,\dots,z_{n-1})$.
\end{claim}
\begin{proof}[Proof of the claim] We denote $(z_1,\dots,z_{n-2})$ by $x$, and $z_{n-1}$ by $z$.\\
If $U$ is an open compact subgroup of $Z_{n-1}(\o)$, 
such that $\tilde{E}$ and the homomorphism $a \in Z_{n-1} \mapsto M(a) \in G_q(\mathbb{C})$ are $U$ invariant, we denote by $u_1,\dots,u_s$ the representatives of 
$Z_{n-1}(\o)/U$, and by $\omega$, the canonical generator of $Z_{n-1}/Z_{n-1}(\o)$.
We put $M''=\max_{i,j}(N_{u_i},N_{\omega})$.\\
 Then for $z$ in $\{z\in Z_{n-1}, |z|_F\leq q_F^{-M''}\}$, and $a=\omega^r u_i u$ in  $\{z\in Z_{n-1}, |z|_F\leq 1\}$ 
(with $u$ in $U$, and $r\in \mathbb{N}$), we have 
$$\tilde{E}(x,za)=\tilde{E}(x,z\omega^r u_i)= \d_{U_n}^{1/2}(u_i){^t\!M(u_i)}\tilde{E}(x,z\omega^r)$$ 
because $z\omega^r$ belongs to $\{z\in Z_{n-1}, |z|_F\leq q_F^{-M''}\}\subset \{z\in Z_{n-1}, |z|_F\leq q_F^{-N_{u_i}}\}$.
But if $r\geq 1$, again one has $$\tilde{E}(x,z\omega^{r})=\d_{U_n}^{1/2}(\omega){^t\!M(\omega)}(\omega_i)\tilde{E}(x,z\omega^{r-1}),$$ 
and $z\omega^{r-1}$ belongs to $$\{z\in Z_{n-1}, |z|_F\leq q_F^{-N_2}\}\subset \{z\in Z_{n-1}, |z|_F\leq q_F^{-N_{\omega}}\},$$ 
and repeating this step, we deduce the equality $\tilde{E}(x,za)= \d_{U_n}^{1/2}(a){^t\!M(a)} \tilde{E}(x,z).$ \end{proof}
 
Hence there is an element $z_0$ in $Z_{n-1}$ with $|z_0|_F=q_F^{-M''}$,
such that for every $(z_1,\dots,z_{n-2})$ in $Z_1\times\dots\times Z_{n-2}$, the vector $\tilde{E}(z_1,\dots,z_{n-1})$ is equal to 
$$\d_{U_n}^{1/2}(z_{n-1})^t\!M(z_{n-1})(z_{n-1})[\d_{U_n}^{-1/2}(z_{0}){^t\!M(z_{0}^{-1})}]\tilde{E}(z_1,\dots, z_{n-2},z_{0})$$ 
for any $z_{n-1}$ with $|z_{n-1}|_F\leq1$.\\
Hence the function $\mathbf{1}_{\{|z_{n-1}|\leq1\}}S(z_1,\dots,z_{n-1})$ is equal to 
$$(x_1,\dots,x_q)^t\!M(z_{n-1})(z_{n-1})[\d_{U_n}^{-1/2}(z_{0}){^t\!M(z_{0}^{-1})}]\tilde{E}(z_1,\dots, z_{n-2},z_{0})\d_{U_n}^{1/2}(z_{n-1})\mathbf{1}_{\{|z_{n-1}|\leq1\}}.$$
One proves as for the function $D$, that function $\mathbf{1}_{\{|z_{n-1}|>1\}}(z_{n-1})S(z_1,\dots,z_{n-1})$ is of the form 
$$\sum_{\beta \ \in \ B} S(z_1,\dots,z_{n-2},z_{\beta})\d_{U_n}^{1/2}(z_{n-1}) S_\beta(z_{n-1})$$
 with 
$S_\beta$ in $C_c^\infty(F)$ for some finite set $B$.\\
By induction hypothesis again, applied to the function $(\d_{U_n}^{-1/2} \tilde{E}_i)(z_1,\dots,z_{n-2},z_{0})$ and the function
 $(\d_{U_n}^{-1/2}S)(z_1,\dots,z_{n-2},z_{\beta})$, 
we deduce that the function $S=\mathbf{1}_{\{|z_{n-1}|\leq 1\}}S + \mathbf{1}_{\{|z_{n-1}|>1\}}S$ is a sum of functions of the requested form. The statement follows as the function $W$ equals $D+S$.\end{proof}

\section{$L^2(Z_nN_n\backslash G_n)$ and discrete series}

First we characterise the Whittaker functions which belong to $\int_{N_n\backslash P_n} |W(p)|^2 dp$ in terms of exponents of 
the ``shifted derivatives'' (see \cite{B}, 7.2.). This result has been used in \cite{M}.\\ 
We say that a character of a $F^*$ is positive if its (complex) absolute value, is of the form $|\ |_F^r$ for some positive real $r$.

\begin{thm}\label{l2p}
Let $\theta$ be a nondegenerate character of the group $N_n$, and $\pi$ be a $\theta$-generic representation of $G_n$, 
let $R$ be defined as before Theorem \ref{restorus}, and let $c_{1,k},\dots, c_{r_k,k}$ be the characters of $Z_k$ 
appearing in a composition series of $\tau=\pi^{(n-k)}$ for 
$k\in R$. Then the integral $$\int_{N_n\backslash P_n} |W(p)|^2 dp$$
converges for any $W$ in $\pi$ if and only if all the characters $c_{i_k,k}\d_{U_{k+1}}^{1/2}$ are positive.

\end{thm}
\begin{proof}
Again we prove the stronger statement:
\begin{st}
  Let $(\pi,V)$ be a $P_n$-submodule of $C^{\infty}(N_n\backslash P_n,\theta)$, such that for every $k$ in $R$, the $G_{k}$-module 
$\tau=\pi^{(n-k)}=\Psi^{-}(\Phi^{-})^{n-k-1}(\pi)$ has a composition series such that, on each respective quotient, the central subgroup $Z_{k}$ acts by 
the characters $c_{1,k},\dots, c_{r_k,k}$.
 Then the integral $$\int_{N_n\backslash P_n} |W(p)|^2 dp$$
converges for any $W$ in $\pi$ if and only if all the characters $c_{i_k,k}\d_{U_{k+1}}^{1/2}$.
 \end{st}
 
Suppose first that all the characters $c_{i_k,k}\d_{U_{k+1}}^{1/2}$ are positive. Let $W$ belong to the space of $\pi$, first we notice the equality
$$\int_{N_n\backslash P_n} |W(p)|^2 dp=\int_{N_{n-1}\backslash G_{n-1}} |W(g)|^2 dg.$$

Now the Iwasawa decomposition reduces the convergence of this integral to that of $$\int_{A_{n-1}} |W(a)|^2 \d_{N_{n-1}}^{-1}(a) d^*a$$ 
Using coordinates $(z_1,\dots,z_{n-1})$ (see Lemma \ref{torusdec}) of ${A_{n-1}}$, the function $\d_{N_{n-1}}^{-1}(z_1,\dots,z_{n-1})$ is equal to 
$\prod_{k=1}^{n-2} (\d_{U_{k+1}}\dots\d_{U_{n-1}})^{-1}(z_k)$.\\
According to Theorem \ref{restorus} the function $|W(z_1,\dots,z_{n-1})|^2$ is bounded by a sum of functions of the form 
\[\prod_{k\in R} [c_{l_k,k}\d_{U_k+1}^{1/2}\dots \d_{U_n}^{1/2}](z_k)v^{m_k}(z_k)\phi_k(z_k) \prod_{l\in V}\phi_l(z_l),\] for non negative integers $m_k$, functions $\phi_k$ in $\mathcal{C}_c^\infty(Lie(Z_k))$, and 
functions $\phi_l$ in $\mathcal{C}_c^\infty(Lie(Z_l))$ vanishing at $z_l=0$.\\
Hence our integral will converge if the same is true of the integrals 
$$\int_{A_{n-1}} \prod_{k\in R} [c_{l_k,k}\d_{U_k+1}^{1/2}\dots \d_{U_n}^{1/2}](z_k)v^{m_k}(z_k)\phi_k(z_k) \prod_{l\in V}\phi_l(z_l) dz_1\dots dz_{n-1},$$ i.e. if the integrals 
 $\int_{Z_k} |c_{i_k,k}|(z_k)|c_{l_k,k}|(z_k)\d_{U_n}(z_k) v_F(t_k)^{m_k}\phi_k(t_k)dz_k$ converge for any $k\in R$.\\
But the restriction of $\d_{U_n}$ to $Z_k$ is equal to $\d_{U_{k+1}}$, so the convergence follows from our assertion on the characters $c_{i_k,k}\d_{U_n}^{1/2}$.\\
Conversely, suppose that every $W$ in $\pi_{(0)}$ belongs to the space $L^2(N_n\backslash P_n)$ corresponding to a right invariant measure on $N_n\backslash P_n$.\\
By Iwasawa decomposition, one gets that $\int_{N_n\backslash P_n} |W(p)|^2 dp=\int_{N_{n-1}\backslash G_{n-1}} |W(g)|^2 dg$ is equal 
to $\int_{A_{n-1}\times F_n}|W(ak)|^2 \d_{N_{n-1}}^{-1}(a)d^*adk$ which is greater than $dk(U)\int_{A_{n-1}}|W(a)|^2 \d_{N_{n-1}}^{-1}(a)d^*a$ 
for some compact open subgroup $U$ fixing $W$. In particular the integral $\int_{A_{n-1}}|W(a)|^2 \d_{N_{n-1}}^{-1}(a)d^*a$ converges for any $W$ in $\pi$.\\
This by Fubini's theorem and smoothness of $W$, implies that $\int_{A_{n-2}}|W(a)|^2 \d_{N_{n-1}}^{-1}(a)d^*a$ is finite for any $W$ in $\pi$.
But the restriction of $\d_{N_{n-1}}$ to $A_{n-1}$ is equal to $\d_{N_{n-2}}\d_{U_{n-1}}$, so that the integral 
$\int_{A_{n-2}}|W(a)|^2 \d_{U_{n-1}}^{-1}\delta_{N_{n-2}}^{-1}(a)d^*a$ is finite for $W$ in $\pi$, which by Iwasawa decomposition again, implies that 
$\int_{N_{n-1}\backslash P_{n-1}} |\d_{U_{n-1}}^{-1/2} W(p)|^2 dp$ is finite.\\
The functions $\d_{U_{n-1}}^{-1/2} W$ belong to the space of $\phi^-(\pi)$, hence by induction, all the characters 
$c_{i_k,k}\d_{U_{k+1}}^{1/2}$ are positive for $k\in R$ with $k\leq n-2$. So we only need to check that the characters 
$c_{i_{n-1},n-1}\d_{U_{n}}^{1/2}$ are positive if $n-1\in R$. Suppose that one of them, $c_{1,n-1}\d_{U_{n}}^{1/2}$ for instance, wasn't.\\
Then, taking $v$ nonzero in $\Psi^-(\pi)$ such that $Z_{n-1}$ multiplies $v$ by $c_{1,n-1}$, according to Proposition \ref{Ker} and taking $W$ a preimage of $v$ in 
$\pi$, there is a positive
 integer $N_a$, such that  
$\d_{U_n}^{-1/2}(a)\pi(a)W(g)-c_{1,n-1}(a)W(g)$ is zero whenever for any $g$ in $G_{n-1}$ with $|z(g)|_F\leq q_F^{-N_a}$. 
As in Claim \ref{claim}, this implies that there is a positive integer $N$, such that $W(ag)$ is equal to 
$\d_{U_n}^{1/2}(a)c_{1,n-1}(a)W(g)$ whenever $|z(g)|_F \leq q_F^{-N}$ and $|a|_F\leq 1$. We recall that $W$ doesn't belong to 
$V(U_n,1)$ (otherwise $v$ would be zero), hence according to Proposition \ref{Ker}, there is $g_0$ in $G_{n-1}$ with 
$|z(g_0)|_F\leq q_F^{-N}$, such that $W(g_0)$ is nonzero. We denote by $W_0$ the function $\pi(g_0)W$, and we recall that the integral 
$$\int_{A_{n-1}}|W_0(a)|^2 \d_{N_{n-1}}^{-1}(a)d^*a= \int_{Z_1\times\dots\times Z_{n-1}}|W_0(z_1\dots z_{n-1})|^2 \d_{N_{n-1}}^{-1}(z_1\dots z_{n-1})dz_1\dots dz_{n-1},$$ 
is finite.
Hence the smoothness of $W_0$ and Fubini's theorem imply that the integral  
$$\int_{Z_{n-1}}|W_0(z_{n-1})|^2 \d_{N_{n-1}}^{-1}(z_{n-1})dz_{n-1}=\int_{Z_{n-1}}|W_0(z_{n-1})|^2 dz_{n-1}$$ is finite.
But for $|z_{n-1}|_F\leq 1$, the function $W_0(z_{n-1})$ is equal to $\d_{U_n}^{1/2}(z_{n-1})c_{1,n-1}(z_{n-1})W(g_0)$ with $W(g_0)$ nonzero, hence 
it is square integrable at zero if and only if $\d_{U_n}c_{1,n-1}^2$, thus $\d_{U_n}^{1/2}c_{1,n-1}$ is positive.\\
\end{proof}

\begin{rem}
The last proof more or less contains the following fact (which is more precisely a consequence of an induction, and the last step of the proof):\\ 
For every character $c_{i_k,n-k}$ appearing in a factor series of $\pi^{(n-k)}$ when this space is nonzero, there is $W$ in $V$, such that $W(z_k)$ is equal to 
$[c_{i_k,n-k}\d_{U_{k+1}}^{1/2}\dots \d_{U_{n}}^{1/2}](z_k)$ near zero. Hence this family of characters is minimal in the sense that 
each of them must occur in the expansion given in Proposition \ref{restorus} of some $W$ in $V$.  
\end{rem}

From this we deduce a characterization of the Whittaker functions in $L^2(Z_nN_n\backslash G_n)$.

\begin{cor}\label{l2g}
 Let $\theta$ be a nondegenerate character of the group $N_n$, and $\pi$ be a $\theta$-generic representation of $G_n$ with unitary central character, 
let the $c_{1,k},\dots, c_{r_k,k}$ be the central characters 
appearing in the factor series of $\tau=\pi^{(n-k)}$, for $k\in R$. Then the integral $$\int_{Z_nN_n\backslash G_n} |W(g)|^2 dg$$
converges for any $W$ in $\pi$ if and only if all the characters $c_{i_k,k}$ are positive.
\end{cor}
\begin{proof}
By the Iwasawa decomposition, the integral $\int_{Z_nN_n\backslash G_n} |W(g)|^2 dg$ converges for every $W$ in $W(\pi,\theta)$ if and only if 
the $\int_{A_{n-1}} |W(a)|^2\d_{N_n}^{-1}(a) dg$ converges for every $W$ in $W(\pi,\theta)$.\\
As the character $\d_{N_n}$ restricts to $G_{n-1}$ as $\d_{N_{n-1}}\d_{U_n}$, this integral is equal to 
$$\int_{A_{n-1}} |\d_{U_n}^{-1/2}W(a)|^2\d_{N_{n-1}}^{-1}(a) dg.$$ But this integral converges for any $W$ in $W(\pi,\theta)$ if and only if so does the 
integral $$\int_{N_n\backslash P_n} |\d_{U_n}^{-1/2}W(p)|^2 dp$$ for any $W$ in $W(\pi,\theta)$.\\
By the statement in the proof of theorem \ref{l2p}, applied to $\d_{U_n}^{-1/2}\otimes \pi'$, this is the case if and only 
all the characters $c_{i_k,k}$.  

\end{proof}

Let $P$ be a standard proper parabolic subgroup of $G_n$, $U$ its unipotent radical, and $M$ its standard Levi subgroup. 
If $(\pi,V)$ is a smooth irreducible representation of $G_n$, one calls cuspidal exponent of $\pi$ with respect to 
$P$, a character $\chi$ of the center of $M$ 
such that the characteristic space of the Jacquet module $(V_U)_{\chi,\infty}$ is nonzero. Denoting by $\Delta$ the set of simple roots 
$\{\alpha_1,\dots,\alpha_n\}$ of $G_n$, 
We denote by $P^{\{i_1,\dots,i_t\}}$ the standard parabolic subgroup associated with the set of positive roots 
$\Delta-\{\alpha_{i_1},\dots,\alpha_{i_t}\}$, by $U^{\{i_1,\dots,i_t\}}$ its unipotent radical, by $M^{\{i_1,\dots,i_t\}}$ its standard Levi subgroup, 
which admits as a central subgroup the product $Z_{i_1}\dots Z_{i_t}$.\\
Notice that except for case D, for 
$\{i_1,\dots,i_t\}=\{2\}$, where we used the notation $U_{2,n-2}$ for $U^{\{1,2\}}$, the group $U^{\{k\}}$ is 
what we already denoted by $U_{k,n-k}$ before.\\
We denote by $A_{i_1,\dots,i_t}^-$ the set 
$$\{z_{i_1}\dots z_{i_t} \in Z_{i_1}\dots Z_{i_t}, \ |z_{i_k}|_F\leq 1, \ \text{and} \ |z_{i_1}\dots z_{i_t}|_F<1\}.$$

Theorem 4.4.6 of \cite{C} then asserts that $\pi$ with unitary central character is a discrete series 
representation if and only if, for every standard parabolic subgroup 
$P^{\{i_1,\dots,i_t\}}$, if $\chi$ is a cuspidal exponent of $\pi$ with respect to 
$P^{\{i_1,\dots,i_t\}}$, the restriction of $\chi$ to $A_{i_1,\dots,i_t}^-$ is less than $1$, or equivalently if $\chi$ restricted to 
$Z_{i_1}\dots Z_{i_t}$ is positive.\\  
We also notice that for any $k$, the Jacquet module $V_{U^{\{i_k\}}}$ surjects onto $V_{U^{\{i_1,\dots,i_t\}}}$, and that the 
character $\delta_{U^{\{i_1,\dots,i_t\}}}$ restricts to $Z_{i_k}$ as $\delta_{U^{\{i_k\}}}$, hence if $\chi$ is a cuspidal exponent of $\pi$ 
with respect to $P^{\{i_1,\dots,i_t\}}$, then $\chi_{|Z_{i_k}}$ is the restriction to 
$Z_{i_k}$ of a cuspidal exponent of $\pi$ with respect to $P^{\{i_k\}}$.\\
This implies that $\pi$ irreducible with unitary central character is a discrete series representation 
if and only if the cuspidal exponents of $\pi$ with respect to maximal 
parabolic subgroups $P^{\{i_k\}}$ have positive restriction to $Z_{i_k}$.\\

We call a character $\chi$ of $Z_k$ such that $(V^{(n-k)})_\chi$ (or equivalently $(V^{(n-k)})_{\chi,\infty}$) 
is nonzero an exponent of the derivative $(\pi^{(n-k)},V^{(n-k)})$. Now we recall that we showed in Proposition \ref{weights2}, 
that the $Z_k$ modules $V_{U_{k,n-k}}$ and $V^{(n-k)}$ have the same nonzero weight subspaces. This allows to prove in our four 
cases the following conjecture of Lapid and Mao (\cite{LM}, conjecture 3.5).

\begin{thm}\label{l2'}
 Let $\pi$ be a generic representation of $G_n$ with unitary central character and with Whittaker model $W(\pi,\theta)$, 
then the following statements are equivalent:

\begin{description}
 \item i) The integral $$\int_{N_nZ_n\backslash G_n} |W(g)|^2 dg$$
converges for any $W$ in $W(\pi,\theta)$.
\item ii) All the exponents of the derivatives of $\pi$ are positive.
\item iii) the representation $\pi$ is square-integrable.
\end{description}
\end{thm}

\begin{proof}
By assumption, the exponents of the derivatives of $\pi$ are the characters $c_{i_k,k}$ of corollary \ref{l2g}, hence
 i) $\Leftrightarrow$ ii) is corollary \ref{l2g}.\\
 ii) $\Leftrightarrow$ iii): we treat the case D separately, so assume first that $G_n$ is not $GSO(2(n-1),F)$.\\
By Proposition \ref{weights}, every cuspidal exponent of $\pi$ corresponding to $V_{U^{\{k\}}}$ is positive if and only if every exponent of the derivative 
$\pi^{(n-k)}$ is positive. But we have already seen that this implies that $\pi$ is a discrete series representation.\\
For the case $D$, we could have reversed the roles of the roots $\alpha_1$ and $\alpha_2$ (which correspond to the two symmetric roots at the end of the Dynkin diagram).
The only effect it would have is to change the definition of the derivative functors $\pi^{(n-2)}$ and $\pi^{(n-1)}$.
 Indeed $U_2$ would become $U_{\alpha_2}$, $Z_1$ and $Z_2$ would be exchanged. The character $\theta_3$ would have to be trivial on $U_{\alpha_2}$ instead of being 
trivial on $U_{\alpha_1}$. But i) and ii) would still be equivalent in this case, and i) is independent of these choices.\\
In both cases, the maps $I_k$ from $V_{U_{k,n-k}}$ to $V^{(n-k)}$ take nonzero weight subspaces 
to nonzero weight subspaces. For $n\geq3$, the space $V_{U_{k,n-k}}$ is equal to $V_{U^{\{k\}}}$. In the first case, $V_{U_{1,n-1}}$ is equal to 
$V_{U^{\{1\}}}=V_{U_{\alpha_1}}$, and it is equal to $V_{U^{\{2\}}}=V_{U_{\alpha_2}}$ in the second case. This implies that 
all the exponents of the derivatives of $\pi$ are positive if and only if all cuspidal exponents of $\pi$ with respect to maximal parabolic subgroups are positive.
Again this proves ii) $\Leftrightarrow$ iii).
 \end{proof}

\end{document}